\documentclass[11pt]{article}
\usepackage{amssymb, amsmath, amsthm, amsfonts, xcolor, enumerate,graphicx}
\usepackage{epstopdf}
\usepackage{epsfig}
\usepackage{fullpage}
\usepackage{hyperref}
\newtheorem{theorem}{Theorem}[section]
\newtheorem{lemma}[theorem]{Lemma}

\newtheorem{thm}[theorem]{Theorem}

\newtheorem{cor}[theorem]{Corollary}

\newtheorem{conj}[theorem]{Conjecture}

\theoremstyle{definition}

\newcommand{\dd}{\partial}
\usepackage{hyperref}

\newcommand{\cP}{\mathcal{P}}

\newcommand{\bR}{\ensuremath{\mathbb{R}}}
\newcommand{\bP}{\ensuremath{\mathbb{P}}}
\newcommand{\bE}{\ensuremath{\mathbb{E}}}

\def\komment#1{}
\let\komment=\footnote

\title{Random $4$-regular graphs have $3$-star decompositions\\ asymptotically almost surely}

\author{
Michelle Delcourt
\thanks{Corresponding Author, Department of Mathematics,
University of Illinois, Urbana, Illinois 61801, USA {\tt delcour2@illinois.edu}. Research supported by NSF Graduate Research Fellowship DGE 1144245.}
\and
Luke Postle
\thanks{Combinatorics and Optimization Department,
University of Waterloo, Waterloo, Ontario N2L 3G1, Canada {\tt lpostle@uwaterloo.ca}. Partially supported by NSERC
under Discovery Grant No. 2014-06162.}}

\begin{document}
\maketitle

\begin{abstract}
\noindent Bar\'{a}t and Thomassen conjectured in 2006 that the edges of every planar 4-regular 4-edge-connected graph can be decomposed into copies of the star with 3 leaves.  Shortly afterward, Lai constructed a counterexample to this conjecture.  Using the small subgraph conditioning method of Robinson and Wormald, 
we prove that a random 4-regular graph has an $S_3$-decomposition asymptotically almost surely, provided the number of vertices is divisible by 3.
\end{abstract}

\section{Introduction}
A question that has garnered much study is whether the edges of a graph $G$ can be decomposed into copies of a small fixed subgraph, say $F$. Of course, some natural divisibility conditions arise for such a decomposition, namely that $e(F)$ must divide $e(G)$. Kotzig observed~\cite{k} that if $G$ is connected and $e(G)$ is even, then $G$ decomposes into copies of $S_2$, the star with $2$ leaves. What happens for larger $F$; in particular, are there natural conditions when $F$ is isomorphic to the $S_3$, the star with $3$ leaves?  
Not much was known about this problem until Thomassen's breakthrough results~\cite{T3} on the weak 3-flow conjecture. In particular, we note the following theorem which follows from a more general theorem of Lov\'{a}sz, Thomassen, Wu, and Zhang~\cite{T13}.

\begin{thm}\label{ThomStar}
If $F\simeq S_k$, the star with $k$ leaves, and $G$ is a $d$-edge-connected graph such that $k$ divides $e(G)$ and $2\le k\le \lceil d/2 \rceil$, then the edge set of $G$ decomposes into copies of $F$.
\end{thm}

In fact, Theorem~\ref{ThomStar} is tight for $k\ge 3$. To see this, first note that if $k > d$, then $K_k$ is a $d$-edge-connected graph with no $S_k$ decomposition.  For $k \le d$ with $k\ge 3$ and $k> \lceil d/2 \rceil$, consider $k$ copies of $K_d$ with edges added so that the resulting graph $G$ is $d$-regular and $d$-edge-connected. If there existed an \emph{$S_k$-decomposition} of $G$ (a decomposition of the edges of $G$ into copies of $S_k$), then because $k > d/2$, such a decomposition would naturally partition the vertices into $\frac{d}{2k} v(G)= \frac{d^2}{2}$ centers of the stars and $\frac{2k-d}{2k}v(G) = \frac{d(2k-d)}{2}$ non-centers. However, the non-centers must form an independent set, and thus, there are at most $k$ of them, the desired contradiction (because $k<\frac{d(2k-d)}{2}$ when $2k-d\ge 2$).   
\begin{figure}
	\centering 
		\includegraphics[scale=0.575]{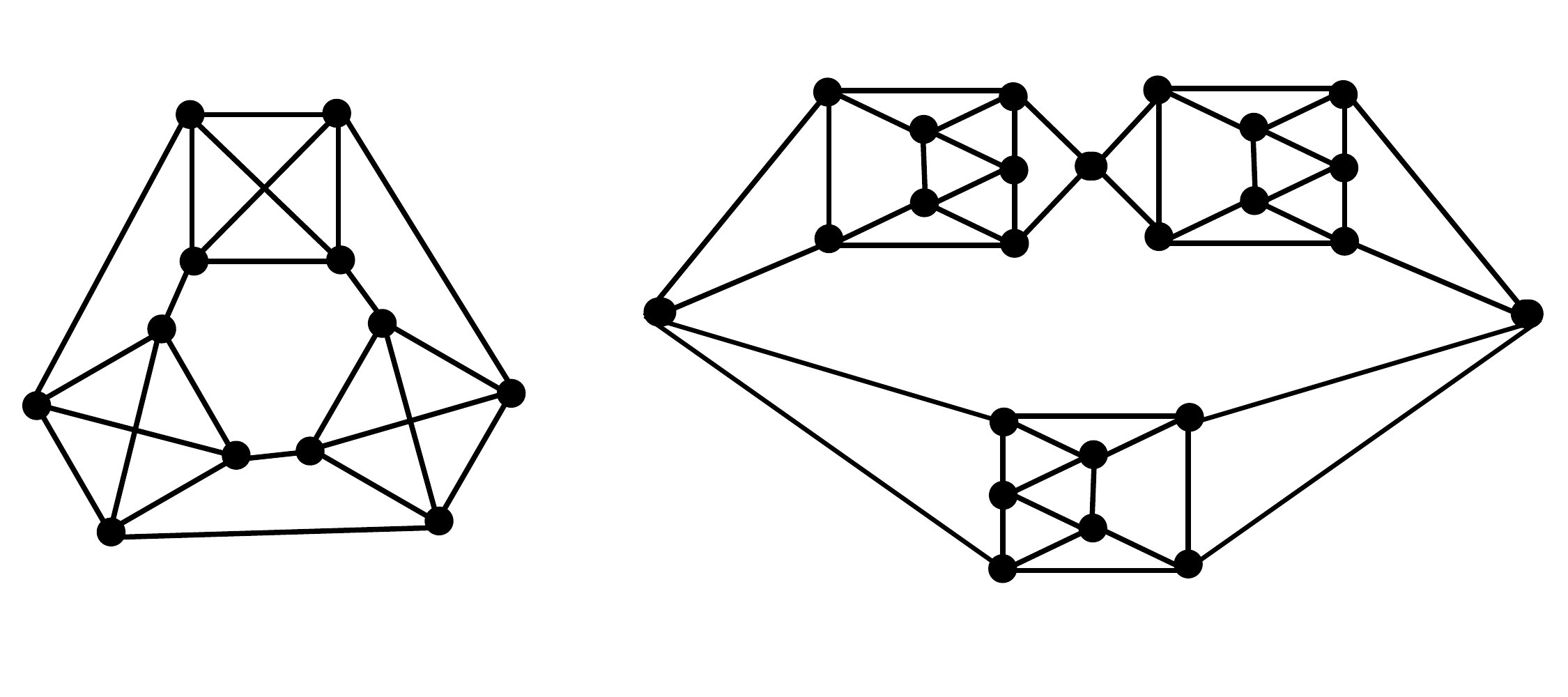}
	\label{F1}
			\caption{On the left is a non-planar $4$-regular $4$-edge-connected graph with no $S_3$-decomposition.  On the right is Lai's planar construction.}
\end{figure}

Thus, when $F$ is isomorphic to $S_3$, Theorem~\ref{ThomStar} implies that a $d$-regular $d$-edge-connected graph $G$ has an $F$-decomposition if $d\ge 5$ and 3 divides $e(G)$. For $d=3$, it is easy to observe that a $3$-regular graph has an $S_3$-decomposition if and only if it is bipartite. As for the case when $d=4$, the construction in Figure~\ref{F1} on the left provides a non-planar example of a $4$-regular $4$-edge-connected graph $G$ where $3$ divides $e(G)$ but $G$ does not have an $S_3$-decomposition. This led Bar\'{a}t and Thomassen~\cite{BT06}, who knew of this example, to conjecture in 2006 that every planar $4$-regular $4$-edge-connected graph $G$ where 3 divides $e(G)$ has an $S_3$-decomposition.  Unfortunately in the following year, Lai presented an infinite family of clever counterexamples (replicated in Figure~\ref{F1} on the right) to this nice conjecture~\cite{L}.  

Given that a typical $d$-regular graph is $d$-edge-connected, a natural setting in which to study these questions is that of random regular graphs. 
We utilize the \emph{configuration model} (also known as the \emph{pairing model}) introduced by Bollob\'{a}s~\cite{B80}.  Let $d \geq 1$ and $dn$ be even; we take a total of $dn$ points and partition them into $n$ \emph{cells} each consisting of exactly $d$ points.  Any perfect matching of $\frac{dn}{2}$ pairs of points is said to be a \emph{configuration}, also known as a \emph{pairing}.  Each configuration corresponds to a multigraph (possibly with loops) where the cells are vertices and the pairs are edges.  We denote the uniform probability space of configurations by $\cP_{n,d}$.  In the configuration model, we choose an element of $\cP_{n,d}$ uniformly at random and discard the result if the corresponding $d$-regular multigraph has loops or parallel edges.  This was shown to be equivalent to choosing a $d$-regular (simple) graph on $n$ vertices uniformly at random (c.f. Wormald's survey paper~\cite{W99} for more details). 

Observe that in any simple $4$-regular graph $G$, an orientation of the edges of $G$ in which every in-degree is either $4$ or $1$ (alternatively every out-degree is either $0$ or $3$) is equivalent to an \emph{$S_3$-decomposition}, that is a decomposition of the edges of $G$ into copies of $S_3$; namely, the vertices with out-degree $3$ are the centers of the stars formed by their out-edges. In light of this, we consider orientations of the edges of a configuration where the out-degree of every cell is 0 or 3, where the \emph{out-degree} of a cell is defined to be the number of points in the cell that are the tail of some edge in the orientation.  We call such an orientation a \emph{$(3,0)$-orientation}. 

The main result of this paper is as follows. Note that all asymptotics in this
article are as $n$ tends to infinity along positive integers divisible by $3$.

\begin{thm}\label{MainThm}
A configuration in $\cP_{n,4}$ has a $(3,0)$-orientation asymptotically almost surely, provided that $n$ is divisible by 3.
\end{thm}

Any 4-regular (simple) graph $G$ on $n$ vertices corresponds to exactly $(4!)^n = 24^n$ configurations in $\cP_{n,4}$.  Because each such graph corresponds to the same number of configurations, it follows that $G$ is a (uniformly) random 4-regular (simple) graph in the configuration model.  The probability that a configuration in $\cP_{n,4}$ is simple tends to a positive constant as $n$ tends to infinity (c.f. Wormald's survey paper~\cite{W99} for more details).  Thus, we have the following as a corollary.  

\begin{cor}\label{Main}
The edges of a random 4-regular (simple) graph on $n$ vertices can be decomposed into copies of $S_3$ asymptotically almost surely, provided that $n$ is divisible by 3.
\end{cor}

Our proof uses the small subgraph conditioning method of Robinson and Wormald~\cite{RW92}. We outline the proof of our main result in Section~\ref{Outline} before proving the remaining required individual lemmas in Sections~\ref{expect}, \ref{var}, and \ref{falling}. However, first we note the connections between this problem and other interesting problems (as well as clarify some notation).

\subsection{Extended History}
There are many connections between orientations and decompositions; of particular interest is the following, known as the circular flow conjecture.

\begin{conj}[Jaeger 1984~\cite{J84}]\label{J1}
Let $k\ge 3$ be odd. Every $(2k-2)$-edge-connected graph $G$ has a \emph{mod $k$-orientation}, that is, an orientation of $E(G)$ such that for every vertex the difference between its out-degree and in-degree is 0 mod $k$. 
\end{conj}

Jaeger proved that his conjecture reduces to the special case of odd regular graphs as follows.
  
\begin{conj}[Jaeger 1988~\cite{J}]\label{J2}
Let $k\ge 3$ be odd. Every $(2k-2)$-edge-connected, $(2k-1)$-regular graph has a \emph{mod $k$-orientation}, that is, an orientation of $E(G)$ in which every in-degree is either $(3k-1)/2$ or $(k-1)/2$.
\end{conj}

Note that when $k=3$, Conjecture~\ref{J1} is actually equivalent to Tutte's nowhere-zero $3$-flow conjecture~\cite{T66}, one of the most famous open problems in structural graph theory.  When $k=5$, Jaeger's conjecture implies the equally famous Tutte's nowhere-zero $5$-flow conjecture~\cite{T54}. Thomassen~\cite{T3} proved Conjecture~\ref{J1} for multigraphs when the edge-connectivity is at least $2k^2+k$. Lov\'{a}sz, Thomassen, Wu, and Zhang~\cite{T13} later improved this and proved Conjecture~\ref{J1} for graphs with edge-connectivity at least  $3k-3$. Despite these massive breakthroughs, proving Jaeger's conjecture still seems intractable. Yet, as noted before, a typical $(2k-1)$-regular graph is $(2k-1)$-edge-connected, and therefore, a natural idea is to study Conjecture~\ref{J2} in the setting of random $(2k-1)$-regular graphs. Using spectral techniques, Jaeger's conjecture was confirmed to hold for random regular graphs provided that $k$ is large enough as follows.  The proof however makes use of the Expander Mixing Lemma and does not apply when $k$ is too small. 

\begin{thm}[Alon and Pra{\l}at 2011~\cite{AP11}]
For large $k$, Jaeger's conjecture holds asymptotically almost surely for random $(2k-1)$-regular graphs.
\end{thm}

Recently, utilizing the small subgraph conditioning method of Robinson and Wormald~\cite{RW92}, Pra{\l}at and Wormald~\cite{PW} were able to confirm Jaeger's conjecture (Conjecture~\ref{J2}) for the case when $k=3$, namely they proved the following theorem.

\begin{thm}[Pra{\l}at and Wormald 2015+~\cite{PW}]
Tutte's nowhere-zero $3$-flow conjecture holds asymptotically almost surely for random $5$-regular graphs.
\end{thm}

Given these results, we were inspired to consider decompositions of random regular graphs, in particular whether the Bar\'{a}t-Thomassen conjecture might hold in the random case. Given Theorem~\ref{ThomStar}, it is also natural to wonder more generally whether random $d$-regular graphs have $S_k$ decompositions for some $k > \lceil d/2 \rceil$. We believe our methods could be applied to these questions. 

As for other subgraphs $F$, Bar\'{a}t and Thomassen~\cite{BT06} conjectured in 2006 that for every tree $T$, there exists $c_T$ such that every $c_T$-edge-connected graph has a decomposition of its edges into copies of $F \simeq T$. Theorem~\ref{ThomStar} confirmed this when $T$ is a star and indeed gives the best possible value of $c_T$ in that case. More recently, Bensmail, Harutyunyan, Le, Merker, and Thomass\'{e}~\cite{BHLMT} proved the conjecture for all trees $T$.  However, determining what the best possible edge-connectivity is or, in the case of random regular graphs, what the best possible degree is, are still open problems.

\subsection{Notation}
Throughout this paper, if $G$ is a multigraph, then we let $V(G)$ and $E(G)$ denote the vertex and edge sets of $G$ respectively whereas $v(G)$ (or $n$) and $e(G)$ denote the number of vertices and edges in $G$.  We say that an event $\mathcal{A} = \mathcal{A}(n)$ holds asymptotically almost surely (a.a.s.) if $\bP[\mathcal{A}(n)] \rightarrow 1$ as $n \rightarrow \infty$ with the obvious necessary parity restrictions on $n$.  For example in the case of finding $S_3$-decompositions of $4$-regular graphs, the necessary condition is that 3 divides $n$ (and hence the number of edges is also divisible by 3).  We denote the set $\left\{1, \ldots, n \right\}$ by $[n]$ and the falling factorial $\frac{n!}{(n-j)!}$ by $(n)_j$.

\section{Outline of the Proof of the Main Result}\label{Outline} 

Let $Y = Y(n)$ denote the number of $(3,0)$-orientations of a random element of $\cP_{n,4}$.  In Section \ref{expect}, we approximate $\bE[Y]$ using Stirling's approximation as follows.
\begin{lemma}\label{Ex}
\[ \bE[Y] \sim \frac{3}{\sqrt{2}}\left(\frac{27}{16}\right)^{n/3}>0.\]
\end{lemma}

In order to show that configurations admit at least one $(3,0)$-orientation, we need to show that a.a.s. $Y > 0$.  It is natural then to try to use the second moment method (coming from Chebyshev's inequality) which says that if $Y$ is a non-negative random variable and $\frac{\bE[Y^2]}{\bE[Y]^2} \rightarrow 0$ as $n\rightarrow \infty$, then a.a.s. $Y>0$. To that end, we approximate $\bE[Y^2]$ in Section \ref{var} using optimization, Taylor expansions and multivariable integration to obtain the following.

\begin{lemma}\label{Ex2}
$$\bE[Y^2] \sim \frac{2\pi n}{9}\cdot \frac{81}{4 \pi n} \sqrt{\frac{3}{2}} \cdot \left(\frac{27}{16}\right)^{2n/3} = \sqrt{\frac{3}{2}} \cdot \frac{9}{2} \left(\frac{27}{16}\right)^{2n/3}.$$
\end{lemma}

Unfortunately, the second moment method does not apply because $\bE[Y]^2$ and $\bE[Y^2]$ are of the same order, as the following corollary notes.

\begin{cor}\label{Ex2b}
$$\frac{\bE[Y^2]}{\bE[Y]^2} \sim \sqrt{\frac{3}{2}} > 0.$$
\end{cor}

Here the distribution of $Y$ is affected by small but not too common (expected number is bounded) subgraphs of the random 4-regular graph, namely short cycles.  In such situations we can attempt to apply the small subgraph conditioning method.  When this method works, by conditioning on the small subgraph counts, we are able to control the variance of $Y$ and in so doing show that $Y>0$ asymptotically almost surely. 

To understand how this works, consider partitioning the set of all 4-regular graphs on $n$ vertices (with $n$ divisible by 3) by the number of triangles.  Within each partition class, the expected number of $(3,0)$-orientations may be smaller than $\bE[Y]$, though by at most a constant factor. Meanwhile the variance inside each class is smaller than the variance of $Y$. Applying the second moment method to the classes individually yields an increase in the probability that $Y>0$, yet it still does not show that this probability tends to 1 asymptotically. So we further partition the classes by the number of $4$-cycles, then by the number of $5$-cycles, and so on. Surprisingly, by conditioning on the numbers of all cycles, it is possible to reduce the variance of $Y$ to any desired fraction of $\bE[Y]^2$. Intuitively, this seems plausible as graphs that have the same number of triangles, $4$-cycles, etc. tend to have a similar structure and so admit less variance in the number of $S_3$-decompositions. Thankfully we do not actually perform such an analysis, relying on the method of Robinson and Wormald~\cite{RW92}; for a proof see Janson's~\cite{S} paper.

%

\begin{thm}[Wormald 1999~\cite{W99}]\label{key}
Let $\lambda_j > 0$ and $\delta_j > -1$ be real numbers, for all $j \geq 1$.  Suppose for each $n$ there are non-negative random variables $X_j = X_j(n)$, $j \geq 1$, and $Y = Y(n)$ (defined on the same probability space) such that 
$X_j$ is integer valued and $\bE[Y]>0$ (for $n$ sufficiently large).    Furthermore, suppose that
\begin{enumerate}
\item[(1.)] for each $j \geq 1,$ $X_1, X_2, \ldots, X_j$ are asymptotically independent Poisson random variables with 
$$\bE[X_i] \sim \lambda_i, \text{ for all } i \in [j];$$ 
\item[(2.)] \vspace{-10pt} $$\frac{\bE\left[ Y(X_1)_{\ell_1} \ldots (X_j)_{\ell_j}\right]}{\bE[Y]} \sim \prod_{i=1}^j \left(\lambda_i\left(1+\delta_i\right) \right)^{\ell_i}$$\vspace{-10pt}
\noindent
for any fixed $\ell_1, \ldots, \ell_j$ where $(X)_\ell$ is the falling factorial;
\vspace{5pt}
\item[(3.)] $$\sum_{i\geq 1}\lambda_i \delta_i^2 < \infty; \text{ and } \frac{\bE[Y(n)^2]}{\bE[Y(n)]^2}\leq \exp\left(\sum_{i\geq 1}\lambda_i \delta_i^2\right) + o(1) \text{ as $n \rightarrow \infty$}.$$
\end{enumerate}
Then, asymptotically almost surely $Y > 0$.
\end{thm}

As in most applications of this method in the literature, we let $X_j$ denote the number of cycles of length $j$ in the multigraph corresponding to a random element of $\cP_{n,4}$.
Here, for $j \geq 1$, $X_1, X_2, \ldots, X_j$ are asymptotically independent Poisson random variables and 
$$\bE[X_j] \sim \lambda_j := \frac{3^j}{2 \cdot j}.$$ This is an immediate consequence of the following theorem of Bollob\'{a}s.

\begin{thm}[Bollob\'{a}s 1980~\cite{B80}]\label{Boll}
For $d$ fixed, let $X_j$ denote the number of cycles of length $j$ in the random multigraph resulting from a configuration in $\cP_{n,d}$.
For $j \geq 1$, $X_1, \ldots, X_j$ are asymptotically independent Poisson random variables with means $\lambda_j = \frac{(d-1)^j}{2 \cdot j}$.
\end{thm}
%

In Section \ref{falling}, we compute $\bE[YX_{j}]$ as follows by extending orientations of small cycles to orientations of the entire graph.
\begin{lemma}\label{Ratio} For all $j\ge 1$,
$$\frac{\bE[YX_{j}]}{\bE[Y]} \sim \frac{3^j}{2 \cdot j} \left(1 + \left( - \frac{1}{3}\right)^j\right)=\lambda_j \left(1 + \left( - \frac{1}{3}\right)^j\right).$$
\end{lemma}

An easy observation from the first examinations of random graphs is that, for any fixed subgraph $H$ with more edges than vertices, a multigraph corresponding to a random element of $\cP_{n,4}$ asymptotically almost surely contains no subgraph isomorphic to $H$.  Informally speaking, we would not expect to have two cycles sharing edges (or for that matter vertices).  Therefore, we concentrate on disjoint cycles and roughly think of them as being independent.  These observations combined with Lemma~\ref{Ratio} imply the following more general form of Lemma~\ref{Ratio}, which computes the joint factorial moments.

\begin{cor}\label{Joint} 
For all $j\ge 1$, if $(\ell_1, \ldots, \ell_j)$ is a sequence of non-negative integers, then
$$\frac{\bE\left[ Y(X_1)_{\ell_1} \ldots (X_j)_{\ell_j}\right]}{\bE[Y]} \sim \prod_{i=1}^j \left(\frac{3^i}{2 \cdot i}  \left(1 + \left( - \frac{1}{3}\right)^i\right) \right)^{\ell_i}=\prod_{i=1}^j \left(\lambda_i \left(1 + \left( - \frac{1}{3}\right)^i\right) \right)^{\ell_i}.$$
\end{cor}


From Lemma~\ref{Ratio}, $\frac{\bE[YX_{j}]}{\bE[Y]} \sim \lambda_j \left(1 + \left( - \frac{1}{3}\right)^j\right)$; thus, we set $\delta_j := - \left(\frac{1}{3}\right)^j > -1$ and verify the following.
\begin{lemma}\label{Verify}
$$\sum_{i\ge 1} \lambda_i \delta_i^2 < \infty \text{ and } \exp\left(\sum_{i \ge 1} \lambda_i \delta_i^2\right) = \sqrt{\frac{3}{2}} \sim \frac{\bE[Y^2]}{\bE[Y]^2}.$$
\end{lemma}
\begin{proof}  Recall that $\lambda_i = \frac{3^i}{2\cdot i}$.  Using that $\sum_{i\ge 1} \frac{x^i}{i} = -\ln(1-x)$ for all $-1 < x < 1$, we obtain that
$$\sum_{i\ge 1} \lambda_i \delta_i^2 = \sum_{i\ge 1} \frac{3^i}{2\cdot i} \cdot \left( - \frac{1}{3}\right)^{2i}
= \frac{1}{2} \sum_{i\ge 1} \frac{\left(\frac{1}{3}\right)^i}{i} 
=\frac{1}{2} \left(-\ln(2/3)\right) < \infty.$$
Thus,
$$\exp\left(\sum_{i \ge 1} \lambda_i \delta_i^2\right) =\exp\left( \frac{1}{2} \left(-\ln(2/3)\right) \right) = \sqrt{\frac{3}{2}}.$$
\end{proof}

Modulo proofs of Lemma~\ref{Ex} (Section~\ref{expect}), Lemma~\ref{Ex2} (Section~\ref{var}), and Lemma~\ref{Ratio} (Section~\ref{falling}), Theorem~\ref{key} now implies our main result as follows.\\

\noindent\emph{Proof of Main Result (Theorem~\ref{Main}).}
Let $3$ divide $n$ and $Y = Y(n)$ denote the number of $(3,0)$-orientations of a random element of $\cP_{n,4}$. Let $X_j$ denote the number of cycles of length $j$ in a random element of $\cP_{n,4}$. We apply Theorem~\ref{key} to $Y$ and $X_j$. Note that (1.) holds by Theorem~\ref{Boll}, (2.) holds by Corollary~\ref{Joint}, and (3.) holds by Lemma~\ref{Verify}. Thus  $Y > 0$ asymptotically almost surely, as desired. \qed

\section{Expected Number of Decompositions}\label{expect}
  
We let $Y = Y(n)$ denote the number of $(3,0)$-orientations of a random element of $\cP_{n,4}$.  We will make use of the following definition.  Given $n$ cells each consisting of 4 points, a \emph{signature} is a set of $2n/3$ points no two belonging to the same cell.  We call these points the \emph{special points} of the signature.  We refer to a cell as a \emph{center} if it contains a special point and as a \emph{leaf} otherwise.  
We say a point is an \emph{in-point} if it is special or in a leaf of the signature and say it is an \emph{out-point} otherwise.

We say a configuration in $\cP_{n,4}$ \emph{extends} a signature if the configuration forms a perfect matching between the in-points and the out-points of the signature.
 We note that a $(3,0)$-orientation extends exactly one signature.  
In this signature, the \emph{centers} correspond to the $2n/3$ cells of out-degree 3 in the orientation (here the special point in each center is the head of the only incoming edge) and the \emph{leaves} correspond to the remaining $n/3$ cells of out-degree 0 in the orientation.  

To prove Lemma~\ref{Ex} though, we switch the order of counting and instead count the number of configurations that extend a given signature. We are now prepared to prove Lemma~\ref{Ex} as follows.\\

\begin{proof}[Proof of Lemma~\ref{Ex}.]
There are a total of ${{n}\choose{2n/3}}4^{2n/3}$ signatures for $n$ cells of 4 points. Recall that a configuration extends a given signature if and only if the configuration matches the in-points of the signature with its out-points. Thus, there are $(\frac{4n}{2})! = (2n)!$ configurations that extend a given signature to a $(3,0)$-orientation. Using Stirling's approximation $s! \sim \sqrt{2\pi s}\left(\frac{s}{e}\right)^s,$ we see that
\[ \bE[Y] = \frac{{{n}\choose{2n/3}} 4^{2n/3}(2n)!}{M(4n)} = 4^{5n/3}\frac{{{n}\choose{2n/3}}}{{{4n}\choose{2n}}}
\sim \frac{3}{\sqrt{2}}\left(\frac{ 3^{3}}{2^{4}}\right)^{n/3} = \frac{3}{\sqrt{2}}\left(\frac{27}{16}\right)^{n/3}>0,\]
\noindent where
\[ M(4n) = \frac{(4n)!}{\left(\frac{4n}{2}\right)!\cdot 2^{(4n)/2}}= \frac{(4n)!}{\left(2n\right)!\cdot 2^{2n}}\] is the number of perfect matchings of $4n$ points.
\end{proof}

\section{The Second Moment Method}\label{var}

In order to calculate $\bE[Y^2]$ for Lemma~\ref{Ex2}, we should count the number of pairs of signatures that a given configuration extends. As in the proof of Lemma~\ref{Ex}, we invert this count by fixing a pair of signatures $\sigma_1$ and $\sigma_2$ and then calculating how many configurations that they both jointly extend. 
To facilitate this count, we consider how the two signatures overlap. One might think that there would be some necessary restriction on how the signatures overlap in order to guarantee the existence of even one configuration that they jointly extend, but strangely this is not the case.

\begin{lemma}\label{FixAB}
For each $A$ and $B$, there are

$$\binom{n}{A,B,\frac{2n}{3}-A-B,  \frac{2n}{3}-A-B, A+B-\frac{n}{3}} \cdot 4^{\frac{4n}{3}-A-B)} \cdot 3^B$$

\noindent
pairs of signatures such that the number of cells that are centers of both signatures with the same special point is $A$ and the number of cells that are centers of both signatures with different special points is $B$.

Furthermore, for each such pair of signatures, there are

$$(3A+2B)!  \cdot (2n-3A-2B)!$$

\noindent
configurations extending both signatures.
\end{lemma}
\begin{proof}[Proof of Lemma~\ref{FixAB}.]
Let $A$ denote the number of cells that are centers in both $\sigma_1$ and $\sigma_2$ and have the same special point. Let $B$ denote the number of cells that are centers in both $\sigma_1$ and $\sigma_2$ and have different special points.  Note that $A + B$ is maximized when all centers in $\sigma_1$ are centers in $\sigma_2$ as well; thus, $A + B \leq \frac{2n}{3}$.  

We see that we may write $\bE[Y^2]$ in terms of $n$, $A$, and $B$ as follows.  Let $C_1$ denote the set of centers of $\sigma_1$ and $C_2$ the set of centers of $\sigma_2$.  Note that $|C_1\cap C_2| = A+B$. Hence $|C_1\setminus C_2| = \frac{2n}{3} - A - B = |C_2\setminus C_1|$. The set of remaining cells is precisely the set of cells that are leaves in both signatures. There are a total of $n-|C_1\cup C_2|= A+B-\frac{n}{3}$ of them. Since this is non-negative, we have that $A + B \geq \frac{n}{3}$. Hence, we see that for each possible value of $A$ and $B$, there are $$\binom{n}{A,B,\frac{2n}{3}-A-B,  \frac{2n}{3}-A-B, A+B-\frac{n}{3}}$$ ways to partition the cells into these types.

There are $4^{2(\frac{2n}{3}-A-B)}$ ways to pick points that are in a leaf in one signature but a special point in the other signature.  There are $4^A$ choices of special points from the centers in both $\sigma_1$ and $\sigma_2$ that have the same special point.  Likewise, there are $(4\cdot 3)^B$ ways to choose special points for the centers of both $\sigma_1$ and $\sigma_2$ with different special points. This proves the first assertion.

For the second assertion, note that a configuration extends both $\sigma_1$ and $\sigma_2$ if and only if the points that are in-points of both $\sigma_1$ and $\sigma_2$ are matched to points that are out-points of both $\sigma_1$ and $\sigma_2$, and the points that are in-points of $\sigma_1$ and out-points of $\sigma_2$ get matched to points that are out-points of $\sigma_1$ but in-points of $\sigma_2$ and vice versa. There are $3A+2B$ points that are out-points of both signatures since $3A$ of them are contained in centers with the same special point and $2B$ of them are contained in centers with different special points. Since there are $2n$ in-points and $2n$ out-points in each signature, it follows that there are $2n - 3A - 2B$ points that are an out-point in $\sigma_1$ and an in-point in $\sigma_2$. Similarly there are $2n-3A-2B$ points that are an in-point in $\sigma_1$ and an outpoint in $\sigma_2$. Hence, there are $3A+2B$ points that are in-points in both signatures. Thus there are $(3A+2B)!$ ways to match the the points that are out-points of both signatures to the points that are in-points of both signatures and $(2n-3A-2B)!$ ways to match the remaining points. The second assertion now follows.  
\end{proof}

\begin{cor}\label{RangeAB}
$$\bE[Y^2] = \sum_{A,B} \frac{(2n)! \cdot n! \cdot 4^{\frac{7n}{3}} \cdot 3^B \cdot (3A+2B)!  \cdot (2n-3A-2B)!  }{(4n)! \cdot 4^{A+B}\cdot A!\cdot B!\cdot (\left(\frac{2n}{3}-A-B\right)!)^2 \cdot \left(A+B-\frac{n}{3}\right)!} $$
\noindent
where $A$ and $B$ are non-negative integers such that $\frac{n}{3} \le A + B \le \frac{2n}{3}$.
\end{cor}
\begin{proof}[Proof of Corollary~\ref{RangeAB}.]
The computation goes as follows. We range over all possibilities of how two signatures may overlap, i.e. we range over $A$ and $B$. Using Lemma~\ref{FixAB} and tidying the formula gives that

\begin{align*}
\bE[Y^2] &=  \frac{1}{M(4n)} \sum_{\text{pairs of signatures}} \#\text{configurations extending both} \\
           &=  \frac{2^{2n}(2n)!}{(4n)!}
           \sum_{A,B} \binom{n}{A,B,\frac{2n}{3}-A-B,  \frac{2n}{3}-A-B, A+B-\frac{n}{3}} \\
&~          \cdot 4^{2(\frac{2n}{3}-A-B)}\cdot 4^A \cdot (4 \cdot 3)^B \cdot (3A+2B)!  \cdot (2n-3A-2B)!  \\
&=           \sum_{A,B} \frac{(2n)! \cdot n! \cdot 4^{\frac{7n}{3}} \cdot 3^B \cdot (3A+2B)!  \cdot (2n-3A-2B)!  }{(4n)! \cdot 4^{A+B}\cdot A!\cdot B!\cdot (\left(\frac{2n}{3}-A-B\right)!)^2 \cdot \left(A+B-\frac{n}{3}\right)!} 
\end{align*}
\end{proof}

It is useful to normalize $A$ and $B$ by letting $a=A/n$ and $b=B/n$. We let $L$ denote the region 

$$L:=\left\{\left(\frac{A}{n},\frac{B}{n}\right) \in \bR^2: A, B \in \mathbb{Z}\cap [0,2n/3] \text{ and } \frac{n}{3} \leq A+B \leq \frac{2n}{3} \right\}.$$ 

Thus the sum in Corollary~\ref{RangeAB} ranges over $L$. We will also need to consider points of $L$ but without the restriction of $A$ and $B$ being integral. Thus we let $R$ denote the region 

$$R:=\left\{\left(a,b\right) \in \bR^2: 0\leq a,b \leq \frac{2}{3} \text{ and } \frac{1}{3} \leq a+b \leq \frac{2}{3} \right\}.$$  

To prove Lemma~\ref{Ex2}, we will apply Stirling's formula to the formula in Corollary~\ref{RangeAB}. Doing so will yield an exponential part and a polynomial part.

To that end, we introduce the two following functions. First for the exponential part, let us define
\begin{align*}
f(a,b) := & b(\ln3 - \ln4) + (2-3a-2b)\ln(2-3a-2b)+ (3a+2b)\ln(3a+2b)- a (\ln a+\ln4)\\
 & - b \ln b - 2\left(\frac{2}{3}-a-b\right)\ln\left(\frac{2}{3}-a-b\right)- \left(a+b-\frac{1}{3}\right)\ln\left(a+b-\frac{1}{3}\right) - \frac{2}{3} \ln4,
\end{align*}
and therefore,
\begin{align*}
e^{f(a,b)} = & \ \frac{3^b \cdot (2-3a-2b)^{(2-3a-2b)} \cdot (3a+2b)^{(3a+2b)}}{4^{\frac{2}{3}+a+b} \cdot a^a \cdot b^b \cdot \left(\frac{2}{3}-a-b\right)^{2\left(\frac{2}{3}-a-b\right)}\cdot \left(a+b-\frac{1}{3}\right)^{\left(a+b-\frac{1}{3}\right)}}.
\end{align*}

Now for the polynomial part, let us define $$g(a,b):=\frac{1}{2 \pi} \cdot \sqrt{ \frac{(3a+2b)\cdot(2-3a-2b)}{  2\cdot (a+\frac{1}{6n})\cdot (b+\frac{1}{6n})\cdot \left(\frac{2}{3}-a-b+\frac{1}{6n}\right)^2\cdot \left(a+b-\frac{1}{3}+\frac{1}{6n}\right)       }    }.$$

We are now ready to apply Stirling's formula to Corollary~\ref{Ex2}, where recall that the formula is $s! = \left(1 + O\left(\frac{1}{s} \right) \right) \sqrt{2\pi s}\left(\frac{s}{e}\right)^s$. In fact, we apply a variant Stirling's formula known as Gosper's formula which is $s! = \left(1 + O\left(\frac{1}{s} \right) \right) \sqrt{\pi (2s+\frac{1}{3})}\left(\frac{s}{e}\right)^s$. We do this because Stirling's formula approximates $0!$ as $0$ instead of $1$, which is unwieldy for division.

\begin{lemma}\label{Ex2AB}
$$\bE[Y^2] = \sum_{(a,b)\in L} S(a,b) \cdot \frac{g(a,b)}{n} \cdot e^{f(a,b) n}$$
where $S(a,b)$ is the error factor arising from the applications of Gospers's formula.
\end{lemma}

\begin{proof}[Proof of Lemma~\ref{Ex2AB}.]
Using Gosper's formula $s! = \left(1 + O\left(\frac{1}{s} \right) \right) \sqrt{2\pi (s+\frac{1}{6})}\left(\frac{s}{e}\right)^s$, we have the following (where $S\left(\frac{A}{n},\frac{B}{n}\right)$ denotes the error factor arising from using Gosper approximations in the calculation below).

Thus,
\begin{eqnarray*}
\bE[Y^2] &=& \sum_{(\frac{A}{n},\frac{B}{n})\in L} \frac{(2n)! \cdot n! \cdot 4^{\frac{7n}{3}} \cdot 3^B \cdot (3A+2B)!  \cdot (2n-3A-2B)!  }{(4n)! \cdot 4^{A+B}\cdot A!\cdot B!\cdot \left(\frac{2n}{3}-A-B\right)! \cdot  \left(\frac{2n}{3}-A-B\right)!\cdot \left(A+B-\frac{n}{3}\right)!} \\
&=& \sum_{(\frac{A}{n},\frac{B}{n})\in L} \scalebox{1.05}{$ {S\left(\frac{A}{n},\frac{B}{n}\right)} \cdot \sqrt{ \frac{ 2^4 \cdot \pi^ 4\cdot (2n+\frac{1}{6})\cdot (n+\frac{1}{6}) \cdot (3A+2B+\frac{1}{6})\cdot(2n-3A-2B+\frac{1}{6})}{2^6 \cdot \pi^ 6\cdot   (4n+\frac{1}{6}) \cdot (A+\frac{1}{6}) \cdot (B+\frac{1}{6}) \cdot (\frac{2n}{3}-A-B+\frac{1}{6})^2 \cdot (A+B-\frac{n}{3}+\frac{1}{6})       }    }  \cdot e^{f\left(\frac{A}{n},\frac{B}{n}\right) n}$} \\
&\sim & \sum_{(a,b)\in L} \scalebox{1.05}{$S\left(a,b\right) \cdot \frac{1}{2\pi n} \cdot \sqrt{ \frac{(3a+2b)\cdot(2-3a-2b)}{  2\cdot (a+\frac{1}{6n})\cdot (b+\frac{1}{6n})\cdot \left(\frac{2}{3}-a-b+\frac{1}{6n}\right)^2\cdot \left(a+b-\frac{1}{3}+\frac{1}{6n}\right)       }    }  \cdot e^{f\left(a,b\right) n}$}\\
&= & \sum_{(a,b)\in L} S\left(a,b\right) \cdot \frac{g\left(a,b\right)}{n}  \cdot e^{f\left(a,b\right) n}. 
\end{eqnarray*}
\end{proof}


\subsection{Multivariate Calculus}
In order to approximate $\bE[Y^2]$, we first need to determine the global maximum of $f$ on the region $L$ and since we use continuous techniques, we will instead find the global maximum of $f$ on $R$.  To approximate the function $f$, we then will take the Taylor expansion of $f$ around the point attaining the global maximum, since the maximum value (as we will show) is unique.  We extend the definition of $f$ continuously to the boundary of $R$ by defining $x \ln x := 0$ when $x = 0$. We prove the following.

\begin{lemma}\label{global}
The global maximum of $f$ on the region $R$ is $2\ln(3) - \frac{4}{3}\ln(4)$.  This value is uniquely achieved at $P_0=(a_0,b_0)=\left(\frac{1}{9},\frac{1}{3}\right)$. Furthermore, the Hessian matrix, $D^2 f(P_0)$, has determinant 81 and is negative definite.
\end{lemma}

\begin{proof}[Proof of Lemma~\ref{global}.]
First we investigate the interior of $R$. To examine the stationary points of $f$, we start by computing the first partials of $f$:
$$\frac{\dd f}{\dd a} = - 3\ln(2-3a-2b) + 3\ln(3a+2b)- \ln a-\ln 4 + 2\ln\left(\frac{2}{3}-a-b\right) -\ln\left(a+b-\frac{1}{3}\right)$$ 
$$\text{and } \frac{\dd f}{\dd b} = \ln3 - \ln4 -2\ln(2-3a-2b) +2\ln(3a+2b) - \ln b+2\ln\left(\frac{2}{3}-a-b\right) -\ln\left(a+b-\frac{1}{3}\right).$$
By setting $\frac{\dd f}{\dd a} =0$, exponentiating both sides, and rearranging, we obtain
\begin{eqnarray}
(3a+2b)^3\left(\frac{2}{3}-a-b\right)^2 = 4a(2-3a-2b)^3\left(a+b-\frac{1}{3}\right).
\label{eq1}
\end{eqnarray}
By setting $\frac{\dd f}{\dd b} = \frac{\dd f}{\dd a}$, exponentiating both sides, and rearranging we obtain 
$3a (2-3a-2b) =b(3a+2b);$ this simplifies to \begin{eqnarray}6a-9a^2=9ab+2b^2.
\label{eq2}
\end{eqnarray} 
Using the quadratic equation to solve (\ref{eq2}) yields 

$$b = -\frac{9a}{4}\pm \frac{1}{4}\sqrt{9a^2 + 48a}.$$

Since the negative solution yields no point in $R$, let us define $b^*(a)= -\frac{9a}{4}+ \frac{1}{4}\sqrt{9a^2 + 48a}.$ Let $h(a) = (3a+2b^*(a))^3\left(\frac{2}{3}-a-b^*(a)\right)^2 - 4a(2-3a-2b^*(a))^3\left(a+b^*(a)-\frac{1}{3}\right).$ Now one can check that the only zeros of $h(a)$ with $0\le a \le \frac{2}{3}$ are $a=0$, $a=\frac{1}{9}$ and $a=\frac{2}{3}$. However, $b^*(0)=0$ and hence $a=0$ yields a common solution $(0,0)$ of (\ref{eq1}) and (\ref{eq2}) that is not in $R$. Thus the only common solutions of (\ref{eq1}) and (\ref{eq2}) that lie in $R$ are $\left(\frac{1}{9},\frac{1}{3}\right)$ and $\left(\frac{2}{3},0\right)$.  We compute that $f\left(\frac{1}{9},\frac{1}{3}\right) = 2\ln(3) - \frac{4}{3}\ln(4) \approx 0.348832$ and $f\left(\frac{2}{3},0\right)=\ln(3) - \frac{2}{3}\ln(4) = \frac{1}{2} \cdot f\left(\frac{1}{9},\frac{1}{3}\right) \approx 0.174416$.

Now we examine the boundary of $R$. As for the corner points, $f\left(\frac{2}{3},0\right) \approx 0.174416$ as above, $f\left(\frac{1}{3}, 0\right) = \ln(3) - \ln (4) \approx - 0.287768$, $f\left(0,\frac{2}{3}\right) = -\frac{1}{3} \ln(3) \approx - 0.366204$ and $f\left(0,\frac{1}{3}\right) = \frac{2}{3} \ln\left(\frac{4}{3}\right) \approx 0.191788$.

So we consider the interiors of the segments of the boundary. First consider the segment $a=0$ and $\frac{1}{3}\leq b \leq \frac{2}{3}$. Note that $\frac{d^2}{db^2} f(0,b) = \frac{2}{1-b} + \frac{1}{b} - \frac{2}{\frac{2}{3} - b} - \frac{1}{b-\frac{1}{3}} = \frac{b + 1}{b(1-b)}- \frac{2}{\frac{2}{3} - b} - \frac{1}{b-\frac{1}{3}}$, which for $\frac{1}{3} < b < \frac{2}{3}$ is at most $\frac{15}{2} - 6-3 < 0$. Hence $f(0,b)$ is concave when $\frac{1}{3} < b < \frac{2}{3}$ and so there is one local maximum in the interior of the segment. This point can be found by setting $\frac{d}{db} f(0,b) = 0$, which occurs at $b=b_1 \approx 0.393226$ at which point $f(0,b_1) \approx 0.253396$.   

Next consider the segment $a = \frac{1}{3} - b$ and $0\leq b \leq \frac{1}{3}$. Here $\frac{d^2}{db^2} f\left(\frac{1}{3} - b, b\right) = \frac{1}{1+b} + \frac{1}{1-b} - \frac{1}{b} - \frac{1}{\frac{1}{3} - b}= \frac{1}{1+b} + \frac{1}{1-b} - \frac{1}{b(1 - 3 b)}$, which for $0 < b < \frac{1}{3}$ is at most $1 + \frac{3}{2} - 12 < 0$. Hence $f\left(\frac{1}{3} - b, b\right)$ is concave when $0 < b < \frac{1}{3}$ and so there is one local maximum in the interior of the segment. This point can be found by setting $\frac{d}{db} f(\frac{1}{3} - b,b) = 0$, which occurs at $b=b_2 \approx 0.280776$ at which point $f\left(\frac{1}{3}-b_2,b_2\right) \approx 0.245950$ .  

Next consider the segment $a = \frac{2}{3} - b$ and $0\leq b \leq \frac{2}{3}$. Once can check that $f\left(\frac{2}{3}-b, b\right)$ is decreasing when $0< b < \frac{2}{3}$ since $\frac{d}{db} f\left(\frac{2}{3}-b, b\right) = \ln\left(\frac{2-3b}{2-b}\right) < 0$ for $0 < b < \frac{2}{3}$. Hence $f\left(\frac{2}{3}-b, b\right)$ is maximized for $0\leq b \leq \frac{2}{3}$ when $b=0$; as above $f\left(\frac{2}{3},0\right)\approx 0.174416$.

Next consider the segment $b=0$ and $\frac{1}{3}\leq a \leq \frac{2}{3}$. One can check that $f(a,0)$ is increasing when $\frac{1}{3} < a < \frac{2}{3}$ since $\frac{d}{da} f(a,0) = \ln\left(\frac{9a^2}{4(2-3a)(3a-1)}\right) > 0$ for $\frac{1}{3} < a < \frac{2}{3}$. Hence $f(a,0)$ is maximized for $\frac{1}{3}\leq a \leq \frac{2}{3}$ when $a= \frac{2}{3}$; as above, $f\left(\frac{2}{3},0\right) \approx 0.174416$.  

Therefore, the unique global maximum occurs at $P_0 = (a_0,b_0) = \left(\frac{1}{9},\frac{1}{3}\right)$. This proves the first assertion.  Note that this corresponds to setting $3a+2b = 2 - 3a-2b$ and therefore $3A+2B = 2n - 3A-2B =n$.  In other words, the number of points that are in-points of one signature and out-points of the other is equal to the number of points that are either in-points of both signatures or out-points of both signatures.  

To compute the Hessian, first we take second partials of $f$ and evaluate at $P_0$:
$$\frac{\dd^2 f}{\dd a^2} = \frac{9}{2-3a-2b}+\frac{9}{3a+2b} - \frac{1}{a} -\frac{2}{\frac{2}{3}-a-b} - \frac{1}{a+b-\frac{1}{3}},$$ 
$$\frac{\dd^2 f}{\dd a \dd b} = \frac{\dd^2 f}{\dd b \dd a} = \frac{6}{2-3a-2b}+\frac{6}{3a+2b} -\frac{2}{\frac{2}{3}-a-b} - \frac{1}{a+b-\frac{1}{3}},$$ 
$$\text{and }\  \frac{\dd^2 f}{\dd b^2} = \frac{4}{2-3a-2b}+\frac{4}{3a+2b} - \frac{1}{b} -\frac{2}{\frac{2}{3}-a-b} - \frac{1}{a+b-\frac{1}{3}}.$$

Note that $$\frac{\dd^2 f}{\dd a^2}\Bigr|_{\substack{P_0}} = -9,\ \ \frac{\dd^2 f}{\dd a  \dd b}\Bigr|_{\substack{P_0}} = \frac{\dd^2 f}{\dd b  \dd a}\Bigr|_{\substack{P_0}} = -6, \ \text{ and } \ \frac{\dd^2 f}{\dd b^2}\Bigr|_{\substack{P_0}} = -13.$$

\noindent Thus, Hessian matrix of $f$ evaluated at $P_0$ is
$$H := D^2f(P_0) = \begin{pmatrix}
-9 & -6\\
-6 & -13
\end{pmatrix}.$$
The determinant of $H$ is $81$ with eigenvalues $-11+2\sqrt{10}$ and $-11-2\sqrt{10}$; thus, $H$ is negative definite (this also implies that $P_0$ must be a local maximum).
\end{proof}


\subsection{Integrating}


Our next lemma, Lemma~\ref{lemma:EXY2}, approximates the sum in Lemma~\ref{Ex2AB} in terms of $\det D^2f(P_0), g(P_0)$ and $e^{f(P_0)n}$. This is done using a Taylor expansion around $P_0$ combined with multivariable Gaussian integrals to calculate $\bE[Y^2]$ more precisely as follows. The theory used in this proof has already been codified into a black box theorem by Greenhill, Janson and Ruci\'{n}ski~\cite{GJR}, specifically Theorem 2.3 and its special case Theorem 6.3. 

Indeed, given that by Lemma~\ref{global}, $f$ has a unique maximum at some interior point of $R$ and the Hessian $D^2f(P_0)$ is strictly negative definite, our Lemma~\ref{lemma:EXY2} follows easily from Theorem 6.3 of~\cite{GJR} by letting $\mathcal{L}=\mathbb{Z}^2, r=2$, $K$ be our set $R$ defined above, $\phi = f$, $\psi = g$ and $\ell_n = 0$ for each positive integer $n$. Nevertheless, for the sake of the reader, we include a proof of Lemma~\ref{lemma:EXY2} for completeness.

\begin{lemma}\label{lemma:EXY2}
$$\bE[Y^2] \sim \frac{2\pi}{\sqrt{|\det D^2f(P_0)|}} \cdot g(P_0) \cdot e^{f(P_0) n}.$$
\end{lemma}

\begin{proof}[Proof of Lemma~\ref{lemma:EXY2}.]
By Lemma~\ref{Ex2AB},
$$\bE[Y^2] = \sum_{(a,b) \in L} S(a,b) \cdot \frac{g(a,b)}{n} \cdot e^{f(a,b) n},$$
where $S(a,b)$ is the error factor arising from the applications of Gosper's formula.

As before, we denote the Hessian matrix of $f$ evaluated at $P_0$ by $H := D^2f(P_0)$. We denote the gradient vector of $f$ evaluated at $P_0$ by $D := D f(P_0)  =[0,0].$ We integrate near the maximum using a second-order Taylor series expansion. Let $[P-P_0]$ denote a row vector with components $[(a-a_0), (b-b_0)]$, and let $[P-P_0]^T$ be the transpose, a column vector.  This gives that $f(a,b)=f(P)$ near $P_0$ is

\begin{align*}
f(P) &= f(P_0) + D[P-P_0]^T+ \frac{1}{2}[P-P_0]H[P-P_0]^T + O\left(\|P-P_0 \|^3 \right)\\
 &= f(P_0)  + \frac{1}{2}[P-P_0]H[P-P_0]^T + O\left(\|P-P_0 \|^3 \right).
\end{align*}
By Taylor's Theorem, we note that the error is valid provided that $\|P-P_0 \| =o(1)$. 

Let $R^\prime = \left\{P \in R: \|P-P_0 \|\leq n^{-2/5}\right\}$ and $L^\prime = L\cap R^\prime$. Note that for all $P \in R^\prime$, $S(P)g(P) \sim g(P_0)$ because ${S\left(P\right)} \sim 1$ and $\|P-P_0 \|^3 \leq n^{-6/5}$. Thus,
\begin{align*}
\sum_{P \in L^{\prime}} S(P) \cdot \frac{g(P)}{n} \cdot e^{f(P) n} &\sim \frac{g(P_0)}{n} \cdot \sum_{P \in L^{\prime}} e^{f(P) n}\\
&\sim \frac{g(P_0)}{n} \cdot e^{f(P_0) n} \cdot \sum_{P \in L^{\prime}} \exp\left({\frac{1}{2}[P-P_0]H[P-P_0]^T n}\right) \cdot e^{n O\left(\|P-P_0 \|^3 \right)}\\
&\sim \frac{g(P_0)}{n} \cdot e^{f(P_0) n} \cdot \sum_{P \in L^{\prime}} \exp\left({\frac{1}{2}[P-P_0]H[P-P_0]^T n}\right),
\end{align*}

\noindent where the last part follows since $e^{n \cdot O\left(\|P-P_0 \|^3 \right)} = e^{n \cdot O(n^{-6/5})} = e^{O(n^{-1/5})}$ goes to $1$ as $n \rightarrow \infty$.

We note then that if we divide the sum by a factor of $n^2$, then this becomes a Riemann sum over $R^\prime$. That Riemann sum in turn approximates an integral as $n\rightarrow \infty$. Hence

\begin{align*}
\sum_{P \in L^{\prime}} \exp\left({\frac{1}{2}[P-P_0]H[P-P_0]^T n}\right)
&\sim 
n^2 \cdot \int \int_{P \in R^{\prime}} \exp\left({\frac{1}{2}[P-P_0]H[P-P_0]^T n}\right) dP.
\end{align*}

Now we change variables by letting $x=(a-a_0)\sqrt{n}$ and $y=(b-b_0)\sqrt{n}$ for $P=(a,b)$. Note that the region of $x,y$ corresponding to $R^\prime$ approaches the whole real plane since the original side length of the box $R'$ is $n^{-2/5}$ and hence the scaled region has side length $n^{-2/5}\sqrt{n}=n^{1/10}$ which goes to $\infty$ as $n \rightarrow \infty$. Thus, this change of variable transforms the integral into 

$$ \int \int_{P \in R^{\prime}} \exp\left({\frac{1}{2}[P-P_0]H[P-P_0]^T n}\right) dP \sim \frac{1}{n} \cdot \int \int_{\mathbb{R}^2} \exp\left([x, \  y]\frac{H}{2}[x, \  y]^T\right) dx dy,$$
where $H$ is the Hessian matrix of $f$ evaluated at $P_0$. Diagonalizing and using the Gaussian integral, that is $\int_{-\infty}^{\infty} e^{-x^2} dx = \sqrt{\pi}$, we see that the integral evaluates to

$$ \frac{1}{n} \cdot \sqrt{\frac{\pi^2}{|\det \frac{H}{2}|} } = \frac{1}{n} \cdot \sqrt{\frac{4\cdot\pi^2}{|\det H|} }  = \frac{2\pi}{n\sqrt{|\det H|}}.$$
Therefore,
$$\sum_{P \in L^{\prime}} S(P) \cdot \frac{g(P)}{n} \cdot e^{f(P) n} \sim \frac{2\pi}{\sqrt{|\det H|}} \cdot g(P_0) \cdot e^{f(P_0) n}.$$

Because $H$ is negative definite, the value of $f$ on the boundary of $R^{\prime}$  is $f(P_0) - \Omega(n^{-2})$. However, $f$ is independent of $n$ and $P_0$ is a global maximum. Thus, 
$$\max_{P \in R \backslash R^{\prime}} f(P) = f(P_0) - \Omega(n^{-2}).$$
Observe that
\begin{align*}\bE[Y^2] &= \sum_{(a,b) \in L} S(a,b) \cdot \frac{g(a,b)}{n} \cdot e^{f(a,b) n}\\ &= \sum_{(a,b) \in L \backslash L^{\prime}} S(a,b) \cdot \frac{g(a,b)}{n} \cdot e^{f(a,b) n}+\sum_{(a,b) \in L^{\prime}} S(a,b) \cdot \frac{g(a,b)}{n} \cdot e^{f(a,b) n}\\
&\sim \sum_{(a,b) \in R \backslash R^{\prime}} S(a,b) \cdot \frac{g(a,b)}{n} \cdot e^{f(a,b) n} + \frac{2\pi}{\sqrt{|\det H|}} \cdot g(P_0) \cdot e^{f(P_0) n}.
\end{align*}

Now consider $P=(a,b)\in L\backslash L^\prime$. Since $L\backslash L^\prime \subseteq R\backslash R^\prime$, we have that $$e^{f(a,b)n} = e^{f(P_0)n} \cdot \exp\left(-\Omega\left(n^{1/2}\right)\right).$$ Yet $S(a,b) = O(1)$ and $g(a,b) = O(n^{5/2})$ as each of the terms in the denominator of $g$ are $O(n)$. Thus for each ${(a,b) \in L \backslash L^{\prime}}$, we see that $$S(a,b) \cdot \frac{g(a,b)}{n} \cdot e^{f(a,b) n} = e^{f(P_0)n} \cdot \exp\left(-\Omega\left(n^{1/2}\right)\right).$$

Note that as there are only a polynomial number of points in $L\backslash L^\prime$, namely at most $n^2$, the sum over points in $L\backslash L^\prime$ is also $e^{f(P_0)n} \cdot \exp\left(-\Omega\left(n^{1/2}\right)\right)$.  Therefore, $\bE[Y^2] \sim \frac{2\pi}{\sqrt{|\det H|}} \cdot g(P_0) \cdot e^{f(P_0) n},$ as desired.
\end{proof}

We are ready to prove Lemma~\ref{Ex2}.

\begin{proof}[Proof of Lemma~\ref{Ex2}.]
By Lemma~\ref{lemma:EXY2},
$$\bE[Y^2] \sim \frac{2\pi}{\sqrt{|\det D^2f(P_0)|}} \cdot g(P_0) \cdot e^{f(P_0) n}.$$
Note that $\det D^2f(P_0) = 81$ by Lemma~\ref{global}. Moreover, $$g(P_0) = \frac{1}{2\pi} \cdot \sqrt{\frac{1\cdot 1}{2\cdot\frac{1}{3}\cdot\frac{1}{9}\cdot \left(\frac{2}{9}\right)^2\cdot\frac{1}{9}} } = \frac{1}{2\pi} \cdot \sqrt{\frac{3 \cdot 9^4}{8}} = \frac{81}{4 \pi} \cdot \sqrt{\frac{3}{2}}.$$
Finally, $f(P_0) = 2\ln(3) - \frac{4}{3}\ln(4)$. Hence 
$$\bE[Y^2] \sim  \frac{2\pi}{9} \cdot g(P_0) \cdot e^{f(P_0)n} = \frac{2\pi}{9}\cdot \frac{81}{4 \pi}\cdot \sqrt{\frac{3}{2}} \cdot \left(\frac{27}{16}\right)^{2n/3} = \sqrt{\frac{3}{2}} \cdot \frac{9}{2} \left(\frac{27}{16}\right)^{2n/3}.$$ 
Therefore, $\frac{\bE[Y^2]}{\bE[Y]^2} \sim \sqrt{\frac{3}{2}}.$
\end{proof}

\section{Joint Factorial Moments}\label{falling}

%
%
%



In this section, 
we prove Lemma~\ref{Ratio}.  By definition,
\[
\bE[YX_{j}]  = \frac{1}{M(4n)}\sum_{{j}\text{-cycle } C}  (\#\text{ orientations of cycle } C)\cdot  (\#\text{ extensions of orientations of } C) .
\] 

Note that this is equivalent to
\[
\bE[YX_{j}]  = \frac{1}{M(4n)}\sum_{ \text{oriented } {j}\text{-cycle } C}  \#\text{ extensions of orientations of } C .
\] 

By counting how many configurations extend such oriented cycles, we will prove the following.

\begin{lemma}\label{JointFormula}
The number of oriented cycles with $s$ sinks and $s$ sources is
$$\frac{(n)_{j}}{j} \binom{j}{2s} (4\cdot 3)^{j},$$
\noindent
while the number of extensions to $(3,0)$-orientations for any oriented cycle is
$$ \binom{n-j}{\frac{2n}{3}-j+s} 4^{\frac{2n}{3} - j+s} \cdot 2^s \left(2n-j\right)!$$
\noindent
Therefore, we have that
$$\bE[YX_{j}] = \frac{(n)_{j}}{{M(4n) \cdot j}} 4^{2n/3} 3^{j} (2n-j)! \sum_{s=0}^{\left\lfloor {j}/2\right\rfloor} \binom{{j}}{2s} \binom{n-{j}}{\frac{2n}{3}-{j}+s}2^{3s}.$$
\end{lemma}
\begin{proof}

Any oriented cycle must have the same number of sources (vertices with out-degree equal to 2) and sinks (vertices with out-degree equal to 0); an oriented cycle of length $j$ can have $s$ sources, $s$ sinks and $j-2s$ other vertices for some $0 \leq s \leq \left\lfloor \frac{j}{2}\right\rfloor$. 
The number of oriented cycles of length $j$ with exactly $s$ sources and $s$ sinks is
$$\frac{(n)_{j}}{j} \binom{j}{2s} (4\cdot 3)^{j}.$$

To see this, choose a set of $j$ vertices ($\binom{n}{j}$ ways). The number of cyclic permutations of $j$ entries is $\frac{(j-1)!}{2}$, where we divide by $2$ for reversing the cycle and we can pick sources and sinks in $2\binom{j}{2s}$ ways (the sources and sinks must alternate around the cycle), since once the sources and sinks
have been chosen, then the orientation of the whole cycle is determined. Finally, every vertex in the cycle needs to pick two of its 4 points in ordered fashion from the configuration for endpoints of edges in the cycle.
 
Let $C$ be a cycle of length $j$; vertices that are sinks in $C$ cannot have out-degree 3 and so are not centers. All other vertices of $C$ must be centers.  Thus the number of leaves in $C$ is $s$ and the number of centers in $C$ is $j-s$.  The number of extensions is then given by first completing the signature. 

To this end, we choose $\frac{2n}{3} - (j-s)$ of the $n-j$ vertices outside of $C$ to be centers; this can be done in $\binom{n-j}{\frac{2n}{3}-j+s}$ ways. Then we choose a special point for each such center; this can be done in $4^{\frac{2n}{3} - j+s}$ ways. For each source, for the non-cycle edges we must orient one edge out and one edge; this can be done in $2^s$ ways. The points that are ends of edges of the cycle are already matched. To extend this to a $(3,0)$-orientation, we match the remaining in-points and out-points, of which there are $2n-j$ of each; hence this can be done in $(2n-j)!$ ways. Combining all of these together, we find that the number of extensions to $(3,0)$-orientations for any oriented cycle is
$$\binom{n-j}{\frac{2n}{3}-j+s} 4^{\frac{2n}{3} - j+s} \cdot 2^s \left(2n-j\right)!,$$

\noindent
and therefore the whole expression is
\begin{align*}
\bE[YX_{j}] &= \frac{1}{M(4n)}\sum_{s=0}^{\left\lfloor {j}/2\right\rfloor} \frac{(n)_{j}}{{j}} \binom{{j}}{2s} (4\cdot 3)^{j} \binom{n-{j}}{\frac{2n}{3}-{j}+s} 4^{\frac{2n}{3} - {j}+s} \cdot 2^s (2n-j)!\\
&= \frac{(n)_{j}}{{M(4n) \cdot j}} 4^{2n/3} 3^{j} (2n-j)! \sum_{s=0}^{\left\lfloor {j}/2\right\rfloor} \binom{{j}}{2s} \binom{n-{j}}{\frac{2n}{3}-{j}+s}2^{3s}.
\end{align*}

\end{proof}

Recall that
$$\bE[Y]= \frac{{{n}\choose{2n/3}} 4^{2n/3}(2n)!}{M(4n)}.$$ 
We are now ready to prove Lemma~\ref{Ratio} as follows.\\

\begin{proof}[Proof of Lemma~\ref{Ratio}.]
In the computation of $\frac{\bE[YX_{j}]}{\bE[Y]}$, we use the following approximation where $y$ is a constant and $x$ goes to infinity,
\[
\frac{x!}{(x-y)!} 
       \sim 
 \frac{\sqrt{2\pi x} \left(\frac{x}{e}\right)^x}{\sqrt{2\pi (x-y)} \left(\frac{x-y}{e}\right)^{x-y}} 
      \sim
\left(\frac{x}{e}\right)^y \cdot \left( \frac{x}{x-y}\right)^{x-y} 
     =
   \left(\frac{x}{e}\right)^y \cdot \left( 1+\frac{y}{x-y}\right)^{x-y}
     \sim
     x^y.
\]

Thus,
\begin{align*}\frac{\bE[YX_{j}]}{\bE[Y]} &= \frac{n!}{j \cdot (n-j)!} 3^{j} \sum_{s=0}^{\left\lfloor {j}/2 \right\rfloor} \binom{{j}}{2s} \frac{{{n-{j}}\choose{\frac{2n}{3}-{j}+s}}}{{n \choose 2n/3}}\frac{(2n-j)!}{(2n)!}2^{3s}\\ &= \frac{3^j}{j} \sum_{s=0}^{\left\lfloor {j}/2\right\rfloor} \binom{{j}}{2s} \frac{\left(\frac{2n}{3}\right)!}{\left(\frac{2n}{3}-{j}+s\right)!}\frac{\left(\frac{n}{3}\right)!}{\left(\frac{n}{3}-s\right)!}\frac{(2n-j)!}{(2n)!}2^{3s}\\
&\sim \frac{3^j}{j} \sum_{s=0}^{\left\lfloor {j}/2\right\rfloor} \binom{{j}}{2s} \left(\frac{2n}{3}\right)^{j-s} \left(\frac{n}{3}\right)^s\frac{2^{3s}}{(2n)^j}= \frac{1}{j} \sum_{s=0}^{\left\lfloor {j}/2\right\rfloor} \binom{{j}}{2s} 2^{2s}.
\end{align*}

Note that $\binom{{j}}{2s}$ is the coefficient of $x^{2s}$ in $q(x):= (1+x)^j$, so

\begin{align*}\frac{\bE[YX_{j}]}{\bE[Y]} &\sim \frac{1}{j} \sum_{s=0}^{\left\lfloor {j}/2\right\rfloor} \binom{{j}}{2s} 2^{2s}= \frac{1}{j}\cdot \frac{\left(q(2) + q(-2) \right)}{2} = \frac{1}{2\cdot j}\left(3^j + (-1)^j \right)\\
&= \frac{3^j}{2\cdot j}\left(1 + \left(-\frac{1}{3}\right)^j \right) = \lambda_j \left(1 + \left( - \frac{1}{3}\right)^j\right).
\end{align*}
\end{proof}
\section{Acknowledgments}
The authors would like to thank Bernard Lidick\'y for useful discussion and helpful comments. We would also like to thank the anonymous referees for helping to improve the clarity of our exposition.
%
%
%


\begin{thebibliography}{99}

\bibitem{AP11}
N.~Alon and P.~Pra{\l}at, 
\newblock{Modular orientations of random and quasi-random regular graphs},
\newblock{\em Combin. Probab. Comput.}, {\bf 20} (2011), 321--329. 

\bibitem{BT06}
J.~Bar\'at and C.~Thomassen, 
\newblock{Claw-decompositions and Tutte-orientations}, 
\newblock{\em J. Graph Theory}, {\bf 52} (2006), 135--146.

\bibitem{BHLMT}
 J.~Bensmail, A.~Harutyunyan, T.-N.~Le, M.~Merker, and S.~Thomass\'e, 
\newblock{A proof of the Barat-Thomassen conjecture},
\newblock{\em J. Combin. Theory Ser. B}, {\bf 124} (2017), 39--55.

\bibitem{B80}
B.~Bollob\'as, 
\newblock{A probabilistic proof of an asymptotic formula for the number of labelled regular graphs}, 
\newblock{\em European J. Combin.}, {\bf 1} (1980), 311--316.

\bibitem{GJR}
C.~Greenhill, S.~Janson, and A.~Ruci\'{n}ski,
\newblock{On the number of perfect matchings in random lifts},
\newblock{\em Combin. Probab. Comput.}, 19(5-6), 2010, 791--817.

\bibitem{J84}
 F.~Jaeger, 
\newblock{On circular flows in graphs}, 
\newblock{ \em Finite and infinite sets}, Vol. I, II (Eger, 1981),
vol. 37 of Colloq. Math. Soc. J\'anos Bolyai, North-Holland, Amsterdam, 1984, 391--402.

\bibitem{J}
F.~Jaeger, 
\newblock{Nowhere-zero flow problems}, 
\newblock{ \em Selected topics in graph theory}, \textbf{3}, Academic Press, San
Diego, CA, 1988, 71--95.

\bibitem{S} S.~Janson, 
\newblock{Random regular graphs: asymptotic distributions and contiguity}, 
\newblock{ \em Combin. Probab. Comput.}, \textbf{4} (1995), 369--405.
16

\bibitem{k} A.~Kotzig, 
\newblock{From the theory of finite regular graphs of degree three and four}, 
\newblock{ \em Casopis P\'est. Mat.}, \textbf{82} (1957), 76--92 (in Slovak).

\bibitem{L} H.-J.~Lai, 
\newblock{Mod $(2p + 1)$-orientations and $K_{1,2p+1}$-decompositions}, 
\newblock{\em SIAM J. Discrete Math.},
\textbf{21} (2007), 844--850.

\bibitem{T13} L.M.~Lov\'asz, C.~Thomassen, Y.~Wu, and C.-Q.~Zhang, 
\newblock{Nowhere-zero 3-flows and modulo k-orientations}, 
\newblock{ \em J. Combin. Theory Ser. B}, \textbf{103} (2013), 587--598.

\bibitem{PW} P.~Pra{\l}at and N.~Wormald, 
\newblock{Almost all 5-regular graphs have a 3-flow},
\newblock{arXiv:1503.03572.}

\bibitem{RW92}
R.W.~Robinson and N.C.~Wormald, 
\newblock{Almost all cubic graphs are Hamiltonian}, 
\newblock{\em Random Structures Algorithms}, {\bf 3} (1992), 117--125.

\bibitem{T3} C.~Thomassen, 
\newblock{The weak 3-flow conjecture and the weak circular flow conjecture}, 
\newblock{\em J. Combin. Theory Ser. B}, \textbf{102} (2012), 521--529.

\bibitem{T54} W.T.~Tutte, 
\newblock{A contribution to the theory of chromatic polynomials}, 
\newblock{\em Canadian J. Math.}, {\bf 6} (1954), 80--91.

\bibitem{T66} W.T.~Tutte, 
\newblock{On the algebraic theory of graph colorings}, 
\newblock{ \em J. Combinatorial Theory}, {\bf 1} (1966), 15--50.

\bibitem{W99}
 N.C.~Wormald, 
\newblock{Models of random regular graphs}, 
\newblock{\em Surveys in combinatorics}, 1999 (Canterbury), vol. 267 of London Math. Soc. Lecture Note Ser., Cambridge Univ. Press, Cambridge,
1999, 239--298.

%
%
%
%
%
%
%
%
%
%
%
%
%
%

\end{thebibliography}
\end{document}